\documentclass{amsart}
\usepackage{graphicx}
\usepackage{amssymb}
\usepackage{array}
\usepackage{enumerate}
\usepackage{url}
\usepackage{algorithm}
\usepackage{dsfont}

\DeclareMathOperator{\sgn}{sgn} 
\DeclareMathOperator{\Hi}{AH} 
\DeclareMathOperator{\Ai}{AI}

\newtheorem{theorem}{Theorem}
\newtheorem{proposition}[theorem]{Proposition}
\newtheorem{lemma}[theorem]{Lemma}

\newtheorem{corollary}[theorem]{Corollary}
\numberwithin{theorem}{section}
\theoremstyle{remark}
\newtheorem*{remark}{Remark}

\newtheoremstyle{named}{}{}{\itshape}{}{\bfseries}{.}{.5em}{\thmnote{#3}}
\theoremstyle{named}
\newtheorem*{namedlemma}{Lemma}
\theoremstyle{named}
\newtheorem*{namedlemma2}{Lemma}

\makeatletter
\let\@@pmod\pmod
\DeclareRobustCommand{\pmod}{\@ifstar\@pmods\@@pmod}
\def\@pmods#1{\mkern 2mu({\operator@font mod}\mkern 3mu#1)}
\makeatother

\begin{document}

\title[Cubic exponential sums]{Asymptotics and formulas for cubic exponential sums}
\author[G.A. Hiary]{Ghaith A. Hiary}
\thanks{Preparation of this material is partially supported by
the National Science Foundation under agreements No. 
 DMS-1406190. The author is pleased to thank
   ICERM where part of this work was conducted.}
\address{Department of Mathematics, The Ohio State University, 231 West 18th
Ave, Columbus, OH 43210.}
\email{hiary.1@osu.edu}
\subjclass[2010]{Primary 11L03, 11L15; Secondary 11Y35.}
\keywords{Cubic exponential sums, The van der Corput iteration.}

\begin{abstract}
Several asymptotic expansions and formulas for cubic exponential sums are derived. The
expansions are most useful when the cubic coefficient is in a restricted range. 
This generalizes previous results in the quadratic case and helps to clarify how
to numerically approximate cubic exponential sums and how to obtain upper bounds
for them in some cases.
\end{abstract}

\maketitle

\section{Introduction} \label{intro}
Let $a$ denote an integer, $q$ a positive integer, 
$b$ an integer relatively prime to $q$
and $e(x) := e^{2\pi i x}$.
Bombieri and Iwaniec analyzed in \cite{bombieri-iwaniec-1} 
the cubic exponential sum
\begin{equation}\label{cubic sum v1}
\sum_{N<n\le 2N} e\left(\frac{an+b n^2}{q} + \mu n^3\right),\qquad
(q\le N, \quad (b,q)=1, \quad 0<\mu \le N^{-2}).
\end{equation}
This was part of their breakthrough method
to bound the maximal size
of the Riemann zeta function on the critical line.
In view of the importance of these sums
it is of interest to study the generalized sum
\begin{equation}\label{intro eq 4}
\sum_{n=0}^N e\left(\frac{an+b n^2}{2q}\right)e\left(
\alpha n +\beta n^2 + \mu n^3\right),
\end{equation}
where $\alpha,\beta$ and $\mu$ are real numbers. We give asymptotic expansions 
and formulas for this sum 
that are perhaps most useful when the cubic coefficient $\mu$ is small enough, satisfying 
$\mu\ll N^{-2}$. 
Our motivation comes in part from an algorithm to compute the zeta function 
derived in \cite{hiary}
where the essential ingredient was a method for numerically
evaluating sums of the form
\begin{equation}\label{intro eq 3}
\sum_{n=0}^N e\left(\alpha n +\beta n^2 + \mu n^3\right), \qquad \mu
\ll N^{-2}.
\end{equation}
In particular, these asymptotics 
could improve the practicality of this method by enabling the use of an explicit
asymptotic expansion 
instead of precise numerical computations when appropriate. Furthermore, as a by-product
we obtain upper bounds for cubic sums.
These results are influenced by the work of
Bombieri and Iwaniec in \cite{bombieri-iwaniec-1}, and the work of 
Fiedler, Jurkat and K{\"o}rner in \cite{FJK}. The latter
obtained asymptotic expansions 
for quadratic exponential sums 
that yield a rough approximation for such sums 
(typically accurate to within square-root of the length).

 We introduce some notation first.
 Let $[x]:= \lfloor x+1/2\rfloor$ denote the nearest integer to $x$, 
 $\sgn(x) := 1$ or $-1$ according to whether $x\ge 0$ or $x<0$,
and $\mathds{1}_{\mathcal{C}}$ the indicator function of 
whether the condition
$\mathcal{C}$ is satisfied. 
 For 
integer $k$, define $k^* := - k \bar{k}^2$ where $k\bar{k}\equiv 1 \pmod{q}$
subject to the additional condition that $4|\bar{k}$ if $q$ is odd, and
let $\delta := 0$ or $1$ according to whether $bq$ is even or odd. 
These definitions of $k^*$ and $\delta$
come directly from the formula for a complete Gauss sum in \cite[Lemma 1]{FJK}.
Furthermore, let $\delta_1=0$ or $1$ according to whether $bq+a$ is even or odd.

Define the Gauss sum
\begin{equation}
    g(b,q):=\frac{1}{2\sqrt{q}}\sum_{h=0}^{2q-1}e\left(\frac{bh^2}{2q}\right),
\end{equation}
which has modulus $1$ or $0$, according to whether $bq$ is even
or not. It is well-known that this sum has a closed-form evaluation in terms of the Kronecker symbol. 
In addition, define
\begin{equation}
    H_N(\alpha,\beta,\mu) := \sum_{n=0}^N e\left(\alpha n+\beta n^2+\mu
    n^3\right), 
\end{equation}
where if $N$ is negative then the summation range is taken over $N\le n
\le 0$. Note that in contrast to the sum \eqref{intro eq 4}, $H_N(\alpha,\beta,\mu)$ does not
incorporate rational approximations for the linear and quadratic arguments 
explicitly.

We write $h_1(x) = O(h_2(x))$,
or equivalently $h_1(x) \ll h_2(x)$, when there is an absolute constant
$C_1$ such that $|h_1(x)| \le C_1 h_2(x)$ for all values
of $x$ under consideration (which will usually make a set of the form $x\ge x_0$).

Using conjugation if necessary we may restrict to $\mu \ge 0$.
In fact, we will assume $\mu > 0$, otherwise one reduces to 
the quadratic sums already treated in \cite{FJK}. 
With this in mind, the basic 
results are given in Propositions \ref{main theorem}, \ref{main theorem 2} and \ref{intro theorem 1} in Section~\ref{summary
 formula 1}. These propositions  furnish transformation formulas for cubic exponential sums
 including explicit error
 bounds. As special cases we
  give in the next few paragraphs several  formulas and asymptotic expansions
  for cubic sums that are meant to be interesting  specializations.

Theorem~\ref{intro theorem 2} below is a specialization of Proposition~\ref{main
theorem}. The theorem isolates a main term for the cubic sum
$H_N(\alpha,\beta,\mu)$
and says, roughly, that $H_N(\alpha,\beta,\mu)$ splits into the product of an ``arithmetic factor''
which is (mostly) determined by rational approximation to $\alpha$ and $\beta$, 
times an ``analytic factor''  determined by the error in said approximation and by $\mu$ and $N$. 

\begin{theorem}\label{intro theorem 2}
    Suppose 
    $2\beta = b/q +2\eta$ where
    $|\eta| \le 1/(8qN)$ and $0<q\le 4N$
    with $(b,q)=1$,
    $2\alpha = a/q+ 2\epsilon$
    where $-1/(4q)\le \epsilon < 1/(4q)$, and $6\mu q N^2<1$.
Define $u:=[2q(\epsilon-\eta^2/(3\mu))]$
$v:=[2q(\epsilon+2\eta N +3\mu N^2)]$, and let
(i) $\Omega := \{0,v\}$ if $\eta \ge 0$ or $\eta \le -3\mu N$,
(ii) $\Omega := \{0,u,v\}$ if $\eta \ge 0$ or $-3\mu N < \eta <0$.
Then 
 $$H_N(\alpha,\beta,\mu)
     =\sum_{\substack{\ell\in \Omega \\ \textrm{distinct $\ell$}}} D_{\ell}(a,b,q) 
\int_0^N e\left(\epsilon t+\eta t^2+\mu t^3 - \frac{\ell t}{2q}\right)dt + O(\sqrt{q}\log 2q),
     $$
where the arithmetic factor $D_{\ell}(a,b,q)$ is give by
$$D_{\ell} (a,b,q) := \mathds{1}_{\ell\equiv \delta_1\pmod*2}\, \frac{g(b+\delta
    q,q)}{\sqrt{q}}  e\left(\frac{b^* (a+\ell)^2}{8q}\right).$$
\end{theorem}

The Diophantine conditions on $\alpha$ and $\beta$ appearing
in the theorem
can always be fulfilled via the Dirichlet approximation theorem
and using a continued fractions algorithm (though the denominator $q$ that arises
for a generic $\beta$ can be of the same order as $N$).
Hence, the theorem can be applied with any $\alpha$ and $\beta$, provided 
that  $\mu$ is small enough.
If $\eta \ge 0$ or $\eta \le -3\mu N$, on the one hand, 
then exactly one of the $D_{\ell}$ terms can possibly be nonzero.
Moreover, if $\eta \ge 0$ then $\Omega \subset\{0,1\}$, 
while if $\eta \le -3\mu N$ then $\Omega \subset\{0,-1\}$. 
On the other hand, if  $-3\mu N < \eta < 0$ then 
at most two of the $D_{\ell}$ terms can possibly be nonzero and 
$\Omega \subset\{0,1,-1\}$. For example, if $\eta \ge 0$ 
and $v=0$, which is a typical situation, then
\begin{equation}
H_N(\alpha,\beta,\mu) = D_0 (a,b,q)\int_0^N e\left(\epsilon t+\eta t^2+\mu t^3\right)dt + O(\sqrt{q}\log 2q).
\end{equation}
Note that if $\delta_1=0$ then $D_0(a,b,q)=0$, and so there is no main term in this
case. In particular,
$H_N(\alpha,\beta,\mu)\ll \sqrt{q}\log 2q \ll \sqrt{N}\log (N+2)$.

\begin{remark}
If we let $f(x)=\epsilon x+\eta x^2+\mu x^3$ for a minute, then 
in the notation of the next section $u=[2qf'(-\omega)]$, $v=[2qf'(N)]$
and $0=[2qf'(0)]$. 
Also, $g(b,q) = G(0,b;2q)$ and $H_N(\alpha,\beta,\mu) = C(N;0,0,1;f)$. 
\end{remark} 

Theorem~\ref{intro theorem 3} is 
a specialization of Proposition~\ref{main theorem} also, 
and provides a van der Corput type iteration for $H_N$.
Using the periodicity relation $e(z+1)=e(z)$ 
we may restrict  $\alpha$, $\beta$ and $\mu$  to the interval
$[-1/2,1/2)$, where we have $0<\mu $ as before. 
In view of this, the length $|N'|$ of the
 transformed sum will be smaller than the length 
 $N$ of the original sum provided that $\beta$ and $\mu$ are small enough.

\begin{theorem}\label{intro theorem 3}
    Let $N'=[\alpha+2\beta N+3\mu N^2]$.
    Suppose that $\alpha,\beta \in [-1/2, 1/2)$ and $0<6 \mu N^2 < 1$. If
        $|\beta| > 1/N$
then
\begin{equation*}
\begin{split}
    H_N(\alpha,\beta,\mu) &= 
\frac{c_2}{\sqrt{2|\beta|}}H_{N'}\left(\frac{\alpha}{2 \beta}+\frac{3\alpha^2\mu}{8\beta^3},-\frac{1}{4 \beta}-\frac{3\alpha \mu}{8\beta^3},\frac{\mu}{8 \beta^3}\right)\\
&+ O\left(\frac{\mu N^2+\mu^2N^5}{\sqrt{|\beta|}}+
\frac{\mathds{1}_{\beta>0}}{\sqrt{\beta}}+\frac{\mathds{1}_{\beta<0}}{\sqrt{-(\beta+3\mu
N)}}+\log
(|N'|+2)\right), 
\end{split}
\end{equation*}
where
$\displaystyle c_2:=e\left(\frac{\sgn(\beta)}{8}-\frac{\alpha^2}{4\beta}-\frac{\alpha^3\mu}{8\beta^3}\right)$.
\end{theorem}

The term $\mathds{1}_{\beta>0}/\sqrt{\beta}$ in the remainder
arises from estimating boundary
terms $\mathcal{B}$ in Proposition~\ref{main theorem} when $\beta >0$, while
$\mathds{1}_{\beta<0}/\sqrt{-(\beta+3\mu N)}$ arises from estimating 
$\mathcal{B}$ when $\beta <-1/N$; see proof of theorem in Section~\ref{proofs}.
Additionally, one can replace $1/\sqrt{\beta}$ with
$\min(N,1/\sqrt{\beta})$ and similarly for $1/\sqrt{-(\beta+3\mu N)}$. 
Of course, both of these terms can be removed if $\mathcal{B}$ is included explicitly in
the theorem.
The term $(\mu N^2+\mu^2N^5)/\sqrt{|\beta|}$ in the remainder comes from 
estimating derivatives of $H_N(\alpha,\beta,\mu)$ with respect to $\alpha$
trivially.
Therefore, if one is interested in understanding the rough size of
$H_N(\alpha,\beta,\mu)$ rather than derive an asymptotic expansion, 
then it is better 
to bound these derivatives using partial summation, which yields
\begin{corollary}
\begin{equation}\label{H bound}
    \begin{split}
        H^{\max}_N(\alpha,\beta,\mu) &\le 
        c_{\beta,\mu,N}
    H^{\max}_{N'}\left(\frac{\alpha}{2
            \beta}+\frac{3\alpha^2\mu}{8\beta^3},-\frac{1}{4
        \beta}-\frac{3\alpha \mu}{8\beta^3},\frac{\mu}{8 \beta^3}\right)\\        
        &+O\left(
        \frac{\mathds{1}_{\beta>0}}{\sqrt{\beta}}+\frac{\mathds{1}_{\beta
        <0}}{\sqrt{-(\beta+3\mu N)}}+\log
        (|N'|+2)\right).
    \end{split}
\end{equation}
where 
\begin{equation}
    \begin{split}
        &H^{\max}_N(\alpha,\beta,\mu) :=
\max_{N_1 \in [0,N]}\left|\sum_{n=N_1}^N e(\alpha n+\beta n^2+\mu n^3)\right|,\\
&c_{\beta,\mu,N} :=  \frac{1+(c_3\mu N+c_4\mu^2    N^4) |\beta|^{-1}}{\sqrt{2|\beta|}}
\end{split}
\end{equation}
and $c_3$ and $c_4$ are absolute nonnegative constants.
\end{corollary}

Interestingly, one can apply the estimate \eqref{H bound} repeatedly
until one of the conditions required by Theorem~\ref{main
theorem 2} fails. This might yield useful bounds for $H_N$
in some applications.
Also, if desired, the last theorem and corollary 
can both be written in a more symmetric form by using the
change of variable
$\tilde{H}_N(\alpha,\beta,\mu):=H_N(\alpha,\beta/2,\mu)$, which 
enables absorbing
the various powers of $2$ that accompany $\beta$.


Our last example is a corollary of Proposition~\ref{intro theorem 1}. 
This proposition furnishes a transformation formula 
for cubic sums when $\alpha=\beta = 0$ and with rational approximations
included explicitly, which was the type of sum considered 
 in \cite{bombieri-iwaniec-1}. 

\begin{corollary} 
\begin{equation*}
\begin{split}
\sum_{n=0}^N e\left(\frac{an+bn^2}{q}+\mu n^3\right) &\ll 
    \mu^{1/2}N^{3/2}q^{1/2}+\min(N,\mu^{-1/3})q^{-1/2}\\
&+\mu N q^{1/2}+q^{1/2}\log (\mu N^2q+2q).
\end{split}
\end{equation*} 
\end{corollary}

Proofs of Theorems \ref{intro theorem 2} \& \ref{intro theorem 3} are given in Section~\ref{proofs}.
We suggest few improvements to these theorems 
in Section~\ref{improvements}. 
 The remaining sections are devoted for proving
 the propositions in Section~\ref{summary formula 1}.

\section{An initial transformation}\label{initial transformation sec}

Given a sequence of complex numbers $\{a_n\}$ and a set $\mathcal{S}\subset
\mathbb{Z}$, we follow the notation in \cite{FJK} and define
\begin{equation}
\sum_{n\in\mathcal{S}} a_n := \lim_{M\to \infty} \sum_{n=-M}^M a_n
\mathds{1}_{n\in\mathcal{S}},
\end{equation}
where $\mathds{1}_{n\in\mathcal{S}}=1$ if $n\in \mathcal{S}$ and
$\mathds{1}_{n\in\mathcal{S}}=0$ otherwise. Let
$$f(x) := \mu x^3+\beta x^2 + \alpha x$$
where $\alpha$, $\beta$, and $\mu$ are
real numbers. Let $C(N;a,b,q;f)$ denote the cubic exponential sum
\begin{equation}
C(N;a,b,q;f)
:= \sum_{n=0}^N e\left(\frac{an+b n^2}{2q}\right)e(f(n)).
\end{equation}
To analyze this sum we will make heavy use of a truncated Airy--Hardy integral
\begin{equation}
\quad \Hi(\omega,N;\mu,s):=
\int_{\omega}^{\omega+N} e\left(\mu t^3-3s t\right)dt,
\end{equation}
and of the completed integrals
\begin{equation}
\begin{split}
&\Hi(\mu,s) :=\Hi(0,\infty;\mu,s) = \int_0^{\infty} e(\mu t^3-3s t)dt,\\
&\Ai(\mu,s) :=\Hi(\mu,s)+ \overline{\Hi(\mu,s)}
= \int_{-\infty}^{\infty} e(\mu t^3-3s t)dt.
\end{split}
\end{equation}
We will also make use of the complete Gauss sum
\begin{equation}\label{gauss sum}
G(a,b;2q):=\frac{1}{2\sqrt{q}}\sum_{h=0}^{2q-1}
e\left(\frac{ah+bh^2}{2q}\right).
\end{equation}

We begin by applying the Poisson summation 
to decompose $C(N;a,b,q;f)$.  
In doing so, only half of the boundary terms
at $n=0$ and $n=N$ are included, giving the term $(1+B)/2$ in Equation~\eqref{C
formula eq} below.

\begin{lemma}\label{C formula}
Let $\mathcal{E} := \{m\in \mathbb{Z}\,|\, bq+a+m\in 2\mathbb{Z}\}$. Define
\begin{equation}
\omega := \frac{\beta}{3\mu},\qquad
s_m := \mu\omega^2+\frac{m/2-q\alpha}{3q},\qquad
g(m) := \frac{b^*(a+m)^2}{8q},
\end{equation}
and
\begin{equation}
\begin{split}
c_0:=e\left(\frac{2\omega^2\beta}{3}-\omega\alpha\right),
\qquad c_1 = c_0 \, G(0,b+\delta q;2q).
\end{split}
\end{equation}
Also define
\begin{equation}\label{B formula eq}
B :=  e\left(\frac{aN+bN^2}{2q}+f(N)\right),\quad 
B_m:=e\left(\frac{\omega m}{2q}+g(m)\right)\Hi(\omega,N;\mu,s_m).
\end{equation}
Then 
\begin{equation}\label{C formula eq}
C(N;a,b,q;f)= \frac{1+B}{2}+
\frac{c_1}{\sqrt{q}} \sum_{m\in \mathcal{E}} B_m,
\end{equation}
\end{lemma}
\begin{remark}
To avoid notational clutter, we suppressed some parameter dependencies;
e.g.\ we have $s_m=s_m(\alpha,\beta,\mu,q)$, 
$g(m)=g(a,b,q;m)$, and $B=B(N;a,b,q;f)$.
\end{remark}
\begin{proof}
Divide the sum along residue class modulo $2q$, which gives
\begin{equation}\label{intro eq 0}
C(N;a,b,q;f)
= \sum_{h=0}^{2q-1} e\left(\frac{ah+bh^2}{2q}\right)
\sum_{-\frac{h}{2q}\le r \le \frac{N-h}{2q}} e(f(h+2qr)).
\end{equation}
Apply the Poisson summation formula (see e.g.\ \cite[Page 14]{davenport-book}) to 
each inner sum, followed by
 the change of variable $t\leftarrow h+2qt$. The inner sum is
thus equal to
\begin{equation}\label{intro eq 1}
\frac{\mathds{1}_{h=0}+B\,\mathds{1}_{h\equiv N\pmod*{2q}}}{2}+ \frac{1}{2q}\sum_{m\in \mathbb{Z}} e\left(\frac{mh}{2q}\right)\int_0^N
e\left(f(t) - \frac{mt}{2q}\right)dt.
\end{equation}
Substituting \eqref{intro eq 1} into \eqref{intro eq 0},
then recalling the Gauss sum definition \eqref{gauss sum}, we obtain
\begin{equation}\label{intro eq 2}
C(N;a,b,q;f)
=\frac{1+B}{2}+\frac{1}{\sqrt{q}}\sum_{m\in \mathbb{Z}} G(a+m,b;2q) \int_0^N e\left(f(t) - \frac{mt}{2q}\right)dt. 
\end{equation}
Furthermore, by \cite[Lemma 1]{FJK}, 
\begin{equation}\label{G g eq}
G(a+m,b;2q) = 
e(g(m))G(0,b+\delta q;2q)
\mathds{1}_{m\in \mathcal{E}}.
\end{equation}

The integral on the r.h.s.\ of \eqref{intro eq 2} has a saddle point (i.e.\  a point $t\in [0,N]$ 
such that the derivative $f'(t)-m/(2q)=0$)  whenever $-\omega\pm
\sqrt{s_m/\mu}\in [0,N]$, where $\omega = \beta/(3\mu)$.
To isolate the contribution of this saddle point, 
we follow \cite{bombieri-iwaniec-1} expanding $f(t)-mt/(2q)$
around $t=-\omega$. This has the advantage that 
 $f''(-\omega)=0$ and will help unify the subsequent analysis
in terms of the Airy--Hardy integral. 

With this in mind, let $y=t+\omega$, appeal to the identity
\begin{equation}
f(t) - \frac{mt}{2q} = 
\left(\frac{\omega m}{2q}+\frac{2\omega^2\beta}{3}-\alpha \omega \right)
+  \left(\alpha-\frac{m}{2q}-3\mu \omega^2 \right)y
+ \mu y^3,
\end{equation}
which is a Taylor expansion of the l.h.s.\ around $t=-\omega$,
and use the change of variable $y\leftarrow t+\omega$. This leads to
the formula 
\begin{equation}\label{INm int}
\int_0^N e\left(f(t) - \frac{mt}{2q}\right)dt=
e\left(\frac{\omega m}{2q}+\frac{2\omega^2\beta}{3}- \omega\alpha\right)\Hi\left(\omega,N;\mu,s_m\right). 
\end{equation}
Substituting formulas 
\eqref{INm int} and \eqref{G g eq} 
into \eqref{intro eq 2}, and recalling 
the definitions of $B_m$ and $c_1$
immediately yields the lemma. 
\end{proof}

\begin{remark}
The formula on \cite[Page 132]{FJK} gives 
an explicit evaluation of $G(0,b+\delta q;2q)$ in terms of 
the Kronecker symbol. In particular, $|G(0,b+\delta q;2q)|=1$ or $0$.
\end{remark}


\section{Analysis of the transformed sum}
The integral $\Hi(\omega,N;\mu,s_m)$ in \eqref{C formula eq} is treated according
to the following cases.

\begin{enumerate}
    \item If the integrand in $\Hi(\omega,N;\mu,s_m)$ contains one interior saddle point; i.e.\ 
if the derivative $3\mu t^2-3s$ vanishes exactly once over
$t\in (\omega,\omega+N)$, 
then the main term in our evaluation of $\Hi(\omega,N;\mu,s_m)$ 
will given by the completed Airy--Hardy integral $\Hi(\mu,s)$ 
or its conjugate $\overline{\Hi(\mu,s)}$.

\item If there are two interior saddle points, then 
    the main term will be given by the completed Airy integral
    $\Ai(\mu,s)=\Hi(\mu,s)+ \overline{\Hi(\mu,s)}$.

\item If there are saddle points at the edge of the integration interval (i.e.\ at $t=\omega$ or
$t=\omega+N$), then a special treatment is required.

\item Last, in the absence of a saddle point, $\Hi(\omega,N;\mu,s_m)$ 
will be estimated via lemmas 4.2 and 4.3 in \cite{titchmarsh}, or using
integration by parts. 
\end{enumerate}

With this in mind, let
\begin{equation}
\|f'\|^+_N := \max_{0\le x\le N} f'(x),\qquad
\|f'\|^-_N := \min_{0\le x\le N} f'(x)= -\|-f'\|_N^+.  
\end{equation}
(Note that $\|\cdot\|_N^-$ is not a norm
since it does not satisfy the usual triangle inequality,
but a ``reversed'' inequality.)
The quadratic polynomial $f'(x)$ achieves its minimum at
$x=-\omega$. Using this and the earlier assumption $\mu>0$ (so
$f'(x)\to +\infty$ as $x\to \pm \infty$) we deduce that
\begin{equation}\label{M2 formulas}
\|f'\|_N^+ = \left\{\begin{array}{ll}
f'(N),& \textrm{if $\omega \ge - N /2$,}\\
f'(0), & \textrm{if $\omega < -N /2$.}
\end{array}\right.
\end{equation}
Also,
\begin{equation}\label{M1 formulas}
 \|f'\|_N^-= \left\{\begin{array}{ll}
f'(0), &\textrm{if $\omega > 0$,}\\
f'(-\omega),& \textrm{if $-N \le \omega \le 0$,}\\
f'(N),& \textrm{if $\omega < - N$.}
\end{array}\right.
\end{equation}
Note that $\|f'\|_N^-$ and $\|f'\|_N^+$ are continuous in $\omega$.

Let us split the range of summation in \eqref{C formula eq}  
into three intervals determined by the points
\begin{equation}
    M_1 := [ 2q\|f'\|_N^-] \qquad \textrm{and} \qquad M_2 := [2q \|f'\|_N^+],
\end{equation}
noting that by definition $M_1\le M_2$. In addition, we will make use of 
\begin{equation}
M:= M_2-M_1.
\end{equation}

It might enlighten matters at this point to 
refer to the van der Corput iteration in \cite{hiary}.
If $\omega > 0$, then
the derivative $f'(x)$ is strictly increasing on $[0,N]$.
So, taking $a=b=0$ in the cubic sum $C(N;a,b,q;f)$, the van der Corput iteration 
reads
\begin{equation}\label{vdc iteration}
\sum_{n=0}^N e(f(n))=e^{\pi i /4} \sum_{f'(0) \le m \le
f'(N)}\frac{e(f(x_m)-mx_m)}{\sqrt{|f''(x_m)|}}+\mathcal{R}_{N,f}, 
\end{equation}
where $x_m$ is the (unique) solution of $f'(x) = m$ in $0\le x\le N$,
and $\mathcal{R}_{N,f}$ is a remainder term; see \cite{hiary}.
Similarly, if $\omega < -N$ 
then the derivative $f'(x)$ is strictly decreasing on $[0,N]$, 
in which case the iteration \eqref{vdc
iteration}
is modified to have $e^{-\pi i/4}$ (instead of $e^{\pi i /4}$) in front,
and the range of summation becomes $f'(N)\le m\le f'(0)$. In either case, and after
allowing for $a$ and $b$ not necessarily zero, we find that there is a single
saddle point if $m\in (M_1,M_2)$ and no saddle point if $m\not\in [M_1,M_2]$,
with $m=M_1$ or $M_2$ being boundary cases.

In contrast, when $-N < \omega < 0$, the form of the van der Corput iteration 
is significantly different because $f'(x)$ is not strictly monotonic but has a
minimum at $x=-\omega$, 
so both saddle points $-\omega \pm \sqrt{s_m/\mu}$ could fall in $(0,N)$.
Explicitly, if $0\le s_m \le \mu \min\{\omega^2,(\omega+N)^2\}$
then the integral $\Hi(\omega,N;\mu,s_m)$
has two saddle points (counted with multiplicity).
Now, recalling that $s_m= \mu\omega^2+(m/2-q\alpha)/(3q)$, the condition $s_m\ge 0$ 
is met precisely when $m\ge 2qf'(-\omega)=2q\|f'\|_N^-$, and so certainly when $m > M_1$. 
Moreover, since $-N < \omega < 0$ by assumption,
\begin{equation}
\min\{\omega^2,(\omega+N)^2\} = \left\{\begin{array}{ll}
\omega^2,& \textrm{if $\omega \ge -N/2$,}\\
(\omega+N)^2,& \textrm{if $\omega <-N/2$.}
\end{array}\right.
\end{equation}
Hence, the condition $s_m \le \mu \min\{\omega^2,(\omega+N)^2\}$ 
is met precisely when 
\begin{equation}
m\le \left\{\begin{array}{ll}
2qf'(0),& \textrm{if $\omega \ge -N/2$,}\\
2qf'(N),& \textrm{if $\omega < -N/2$.}
\end{array}\right.
\end{equation}
This motivates defining
\begin{equation}\label{M* def}
M^* = \left\{\begin{array}{ll}
[2qf'(0)], & \textrm{if $\omega \ge -N /2$,}\\
\textrm{[} 2qf'(N)],& \textrm{if $\omega < - N /2$.}
\end{array}\right.
\end{equation}
So, if $-N < \omega < 0$ then
the integral $\Hi(\omega,N;\mu,s_m)$ has 
two saddle points if $m\in (M_1,M^*)$, a single saddle point
if $m \in (M^*,M_2)$, and no saddle point if $m\not\in [M_1,M_2]$.

Last, we will use the boundary set
\begin{equation}
\Omega = \{M_1,M^*,M_2\},
\end{equation}
corresponding to the terms in \eqref{C formula eq} that might contain 
a saddle point at the edge.
We observe that if $\omega\not\in (-N,0)$ 
then $M^*=[2q\|f'\|^-_N]$. 
Thus, $M^*=M_1$ and $\Omega = \{M_1,M_2\}$ in this case. Also,
if $\omega = -N/2$ then $M^*=M_2$, and so $\Omega = \{M_1,M_2\}$ 
over a  neighborhood of 
$\omega=-N/2$.

We will use the following lemmas in the sequel. 

\begin{namedlemma}[Lemma 4.2 in \cite{titchmarsh}]
Let $F(x)$ be a real differentiable function such that $F'(x)$ is monotonic, and
$F'(x)\ge m>0$ or $F'(x) \le -m <0$ throughout the interval $[a,b]$. Then
$$\left|\int_a^b e^{iF(x)}dx\right| \le \frac{4}{m}.$$
Moreover, if $G(x)$ is a monotonic function over $[a,b]$ such that
$|G(x)|\le G$ over $[a,b]$ then
$$\left|\int_a^b G(x) e^{iF(x)}dx\right| \le \frac{4G}{m}.$$
\end{namedlemma}

\begin{namedlemma2}[Lemma 4.4 in \cite{titchmarsh}]
Let $F(x)$ be a real twice differentiable function such that 
$F''(x)\ge r>0$ or $F''(x) \le -r <0$, throughout the interval $[a,b]$. Then
$$\left|\int_a^b e^{iF(x)}dx\right| \le \frac{8}{\sqrt{r}}.$$
\end{namedlemma2}

\section{Terms with no saddle point}\label{no saddle pts}

We will have have two treatments for the terms with no saddle point. 
Define the first tail of the sum in \eqref{C formula eq} by
\begin{equation}
\Upsilon_1 := \sum_{\substack{m\in\mathcal{E}\\  m\not\in [M_1-q,M_2+q]}}
e\left(\frac{\omega m}{2q}+g(m)\right)
\Hi( \omega,N; \mu,s_m),
\end{equation}
and the second tail by
\begin{equation}
\Upsilon_2 := \sum_{\substack{m\in\mathcal{E}\\ m\not\in [M_1,M_2]\\ m\in
[M_1-q,M_2+q]}}e\left(\frac{\omega m}{2q}+g(m)\right)
\Hi(\omega,N;\mu,s_m).
\end{equation}
In lemmas \ref{large m lemma}
and \ref{Phi lemma} we bound $\Upsilon_1$, 
 and in Lemma~\ref{medium m lemma} we bound $\Upsilon_2$. 
In both cases, we will use the following integration by parts formula: Let
\begin{equation}
\phi_x(\mu,s) :=\frac{1}{6\pi i}
\frac{e(\mu x^3-3s x)}{\mu x^2-s}
\end{equation}
then 
\begin{equation}\label{large m lemma eq 1}
\Hi(u,v;\mu,s) =\phi_{u+v}(\mu,s) -\phi_u(\mu,s) 
+ \frac{1}{3\pi i}\int_u^{u+v}
\frac{\mu t e(\mu t^3-3s t)}{(\mu t^2-s)^2}dt 
\end{equation}
provided that $\mu t^2-s\ne 0$ on $t\in [u,u+v]$.
Starting with $\Upsilon_1$, and taking $u=\omega$ and $v=N$ in the above
formula, we are motivated to write 
$\Upsilon_1 = \tilde{\Phi}_{\omega+N} -\tilde{\Phi}_{\omega} +
(\Upsilon_1-\tilde{\Phi}_{\omega+N}+\tilde{\Phi}_{\omega})$,
where, after simplification, 
\begin{equation}\label{Phi omega}
\tilde{\Phi}_{\omega+x}=
\frac{qe(f(x))}{\pi i c_0}
\sum_{\substack{m\in\mathcal{E}\\ m\not\in [M_1-q,M_2+q]}}
\frac{e(g(m)-xm/2q)}{2qf'(x)-m},
\end{equation}
and the summation is done by pairing the terms for $m$ and $-m$ whenever possible (as was decreed in Section~\ref{initial transformation sec}). 
This sum is convergent at $x=\omega$ and $x=\omega+N$, which is seen on dividing the
sum 
along residue classes modulo $2q$ and using
the periodicity of $e(g(m))$ and $\mathds{1}_{m\in \mathcal{E}}$ modulo $2q$;
see the proof of Lemma~\ref{Phi lemma} for details.

\begin{lemma}\label{large m lemma}
$\displaystyle \Big|\Upsilon_1 -\tilde{\Phi}_{\omega+N}+\tilde{\Phi}_{\omega}\Big| \le
\frac{128(2|\beta|+3\mu N)}{\pi^2}
\frac{q^3(2q+9)}{(2q+1)^3}$. 
\end{lemma}
\begin{proof}
If $m\not \in [M_1,M_2]$,  
then $\mu t^2-s_m$ does not vanish over $t\in [\omega,\omega+N]$, which
is seen on noting that $6q(\mu(\omega+t)^2-s_m) = 2qf'(u)-m$ and using the
definitions of $M_1$ and $M_2$. 
Therefore, we can apply the integration by parts formula \eqref{large m lemma eq
1} to $\Hi(\omega,N;\mu,s_m)$. The function $1/(\mu t^2-s_m)^2$ that
arises 
is monotonic  over $\omega \le t\le 0$ and $0\le t\le \omega +N$, separately.
This is seen on noting that 
the derivative $\frac{d}{dt} (\mu t^2-s_m)^2 = 4\mu t(\mu t^2-s_m)$ has a single root at $t=0$ over
$[\omega,\omega+N]$.
So we can apply partial summation to each of the intervals $[\omega ,0]$ and
$[0,\omega+N]$ in turn.
We thus obtain that 
the integral on the r.h.s.\ of \eqref{large m lemma eq 1} is bounded by
\begin{equation} \label{large m lemma eq 2}
\begin{split}
\frac{2}{3\pi}\max_{\omega\le t\le \omega+N} \frac{1}{|\mu t^2-s_m|^2} 
&\left(\max_{\omega\le \omega_1<\omega_2\le 0}\left|\int_{\omega_1}^{\omega_2}
\mu t e(\mu t^3-3s_mt)dt\right|\right.\\
&\left. +\max_{0\le \omega_1<\omega_2\le \omega+N}\left|\int_{\omega_1}^{\omega_2}
\mu t e(\mu t^3-3s_mt)dt\right|
\right),
\end{split}
\end{equation}
where the extra $2$ in front is because we consider the real and imaginary part
of $e(\mu t^3-3s_mt)$ separately when applying partial summation.
Using partial summation once again, this time to remove the $t$ from 
each integral in \eqref{large m lemma eq 2}, 
we obtain, on applying Lemma 4.2 in \cite{titchmarsh}, 
that the expression in \eqref{large m lemma eq 2} is 
\begin{equation} \label{large m lemma eq 3}
\le \frac{4}{9\pi^2} \max_{\omega\le t\le \omega+N}\frac{\mu (2|\omega|+N)}{|\mu
t^2-s_m|^3}. 
\end{equation}
Writing $t=\omega + u$ with $0\le u\le N$, and recalling the definitions of
$s_m$ and $\omega$, gives $\mu t^2-s_m = (2qf'(u)-m)/(6q)$. So
\begin{equation}\label{large m lemma eq 5}
\max_{\omega\le t\le \omega+N} \frac{1}{|\mu t^2-s_m|}
= \max_{0\le u\le N} \frac{6q}{|2qf'(u) - m|}.
\end{equation}
Combining  \eqref{large m lemma eq
3}, \eqref{large m lemma eq 5},
and the observation 
$\mu (2|\omega|+N) = 2|\beta|/3+ \mu N$, we see that
the expression in \eqref{large m lemma eq 3} is bounded by
\begin{equation}\label{large m lemma eq 6}
\frac{32}{\pi^2}\max_{0\le u\le N} \frac{q^3(2|\beta|+3\mu N)}{|2qf'(u) - m|^3}.
\end{equation}
Now, by definition,
$M_1-1/2 \le 2q f'(u) \le M_2+1/2$ over
$0\le u\le N$.
Moreover, $m\in \mathcal{E}$ is either always odd or always even. Hence, 
\begin{equation}\label{large m lemma eq 7}
\sum_{\substack{m\in\mathcal{E}\\ m\not\in [M_1-q,M_2+q]}}
\max_{0\le u\le N} \frac{q^3(2|\beta|+3\mu N)}{|2qf'(u) - m|^3}
\le 2(2|\beta|+3\mu N) \sum_{j\ge 0} \frac{q^3}{(q+2j+1/2)^3}.
\end{equation}
We isolate the term corresponding to $j=0$ in the last sum, and note that 
the function $1/(q+2x+1/2)^3$ is decreasing. This gives, on comparing 
the sum to an integral,
that the sum on the r.h.s.\ of \eqref{large m lemma eq 7} is bounded by 
\begin{equation}
\frac{q^3}{(q+1/2)^3}+\int_{0}^{\infty} \frac{q^3}{(q+2x+1/2)^3}dx
=\frac{2q^3(9+2q)}{(2q+1)^3}.
\end{equation}
Substituting this into \eqref{large m lemma eq 7}, then back into \eqref{large m
lemma eq 6}, yields the lemma.
\end{proof}

\begin{lemma}\label{medium m lemma}
$\displaystyle |\Upsilon_2|\le \frac{32}{\pi}q+\frac{8}{\pi}q
\log(2q-1)$.
\end{lemma}
\begin{proof}
Note that $\mu t^2-s_m$ is
monotonic over each of $[\omega,0]$ and $[0,\omega +N]$. 
So we can apply Lemma 4.2 in \cite{titchmarsh} in each interval separately
to deduce that 
\begin{equation}\label{medium m lemma eq 1}
|\Hi(\omega,N;\mu,s_m)| \le 
\frac{4}{\pi} \max_{\omega\le t\le \omega+N} \frac{1}{3|\mu t^2-s_m|}.
\end{equation}
Write $t=\omega+u$, where $0\le u \le N$, then proceed as in  
the proof of Lemma~\ref{large m lemma}
to arrive at the same formula \eqref{large m lemma eq
5}, and ultimately the estimate
\begin{equation}\label{medium m lemma eq 2}
|\Upsilon_2|\le  \frac{8}{\pi}\sum_{0\le j\le (q-1)/2}\frac{2q}{(2j+1/2)}.
\end{equation}
Next, we isolate the term corresponding to $j=0$, and bound the remaining sum by an
integral, i.e.\ we obtain the bound 
\begin{equation}
4q+2q\int_0^{(q-1)/2} \frac{1}{2x+1/2}dx= 4q+q\log(2q-1).
\end{equation}
Substituting this into \eqref{medium m lemma eq 2} 
proves the claim.
\end{proof}

We now consider the sizes of $\tilde{\Phi}_{\omega}$ and 
$\tilde{\Phi}_{\omega+N}$.
To this end, let us introduce the quantity
\begin{equation}
    M_{\max} = 2q|f'(0)|+2q|f'(N)|+|M_1|+|M_2|+q,
\end{equation}
which will serve to ``symmetrize'' the summation interval below.
This choice of $M_{\max}$ is a little arbitrary
since we need only ensure that
$M_{\max} \ge |M_1-q|$ and $M_{\max}\ge |M_2+q|$.

\begin{lemma}\label{Phi lemma}
$\displaystyle \left|\tilde{\Phi}_{\omega+N}\right|
+\left|\tilde{\Phi}_{\omega}\right| \le
\frac{28q}{\pi} + \frac{4q}{\pi}
\log\left(\frac{M_{\max}}{q+1/2}+1\right)$.
\end{lemma}
\begin{proof}
    Recalling the formula \eqref{Phi omega} for $\tilde{\Phi}_{\omega+x}$, we wish to replace the 
summation condition $m\not\in [M_1-q,M_2+q]$ in this formula 
by the symmetric condition $|m|> M_{\max}$. 
Note that $f'(0)=\alpha$, $e(g(m))$ is periodic modulo $2q$, 
and $\mathds{1}_{m\in \mathcal{E}}$ is periodic modulo $2$. 
So dividing the sum along residue classes modulo $2q$, we obtain
\begin{equation}\label{Phi lemma eq 3}
\begin{split}
\tilde{\Phi}_{\omega} &=  
\frac{q}{\pi i c_0}\sum_{h=0}^{2q-1} e(g(h)) 
\sum_{\substack{m\in\mathcal{E}\\ m>M_{\max}\\ m\equiv h\pmod*{2q}}}
\left(\frac{1}{2q\alpha-m}+\frac{1}{2q\alpha+m}\right)\\ 
&+\frac{q}{\pi i c_0}\sum_{h=0}^{2q-1} e(g(h)) \sum_{\substack{m\in\mathcal{E}\cap \mathcal{T}\\ m\equiv h\pmod*{2q}}}
\frac{1}{2q\alpha - m}
\end{split}
\end{equation}
where $\mathcal{T} = [-M_{\max},M_1-q) \cup (M_2+q,M_{\max}]$. 

We start by bounding the double sum on the second line of \eqref{Phi lemma eq 3}.
To this end, if $m> M_2+q$, then, depending on the correct parity of $m$, 
either $m = M_2+q+2m'+1$ or $m=M_2+q+2m'+2$ for some nonnegative integer 
$m'$. Additionally,
 since $m$ belongs to a fixed residue class modulo 
$2q$, $m'$ must increment by a multiple of $q$ as $m$ progresses, 
say $m'=jq$.
So, considering that $M_2\ge 2qf'(0)-1/2=2q\alpha-1/2$,
we deduce $|2q\alpha-m| \ge q+2jq+1/2$.
By a similar reasoning, if $m<M_1-q$, then $|2q\alpha-m| \ge q+ 2jq+1/2$.
(Here, we used the bound $M_1 \le 2q\alpha +1/2$.)
Therefore, the double sum under consideration is of size
\begin{equation}\label{Phi lemma eq 2}
\le \frac{4q^2}{\pi} \sum_{0\le j \le
\frac{M_{\max}}{2q}}
\frac{1}{q+ 2jq+1/2}.
\end{equation}
We isolate the term corresponding to $j=0$, and compare the tail with
an integral like $\int_0^X 1/(q+2xq+1/2)dx$, which yields that
\eqref{Phi lemma eq 2} is 
\begin{equation}\label{Phi lemma eq 22}
\le \frac{8}{\pi}\frac{q^2}{q+1/2}
+\frac{2q}{\pi}\log\left(\frac{M_{\max}}{q+1/2}+1\right).
\end{equation}

Next, we bound the  double sum on the first line of \eqref{Phi lemma eq 3}.
But first let us derive lower bounds for $M_{\max}-2q\alpha$ 
and $M_{\max}+2q\alpha$. To this end, consider that
as $m$ progresses in 
a fixed residue class modulo $2q$, we have $m \ge M_{\max} +2jq+1$ 
where $j$ steps through the nonnegative integers. 
In addition, since $M_{\max} \ge 2q|\alpha|+M_2+q$ 
and since by definition $M_2 \ge 2q\alpha-1/2$, then $M_{\max}-2q\alpha \ge 2q|\alpha| +q-1/2$. 
Similarly, since $M_{\max} \ge 2q|\alpha|+|M_1|+|M_2|+q$
and   $|M_1|+|M_2|\ge 2q|\alpha|-1/2$,
we deduce that $M_{\max}+2q\alpha \ge 2q|\alpha|+q-1/2$.
Therefore, on simplifying and applying the triangle inequality,
the double sum under consideration is bounded by 
\begin{equation}\label{Phi lemma eq 4}
 \frac{2q^2}{\pi} \sum_{\substack{m\in \mathcal{E}\\ m > M_{\max}\\ m\equiv h\pmod*{2q}}}
\frac{4q|\alpha|}{|(2q\alpha-m)(2q\alpha+m)|}
\le \frac{2q^2}{\pi}\sum_{j\ge 0} \frac{4q|\alpha|}{(2q|\alpha|+q+2jq+1/2)^2}.
\end{equation}
The last expression is estimated by
isolating the term corresponding to $j=0$, and comparing the rest to
the integral $\int_0^{\infty} 4q|\alpha|/(2q|\alpha|+q+2xq+1/2)^2dx$.
Doing so yields the bound 
\begin{equation}\label{Phi lemma eq 5}
\frac{2q^2}{\pi}\left(\frac{4q|\alpha|}{(2q|\alpha|+q+1/2)^2}+\frac{4|\alpha|q}{q+2q^2+4|\alpha|q^2}\right)
\le  \frac{6q}{\pi}.
\end{equation}

Put together, inserting the estimates 
\eqref{Phi lemma eq 5} and \eqref{Phi lemma eq 22} into \eqref{Phi lemma eq 3}
shows that
$\tilde{\Phi}_{\omega}$ is bounded by $1/2$ times the r.h.s.\ expression in the
statement of the lemma. 
The other $1/2$ comes from
$\tilde{\Phi}_{\omega+N}$, which satisfies this same bound as
$\tilde{\Phi}_{\omega}$ as can be seen via the same method employed so far.
\end{proof}

\section{Terms with one saddle point}\label{one saddle pt}

\begin{lemma}\label{saddle m lemma}
If $\omega > 0$, then
\begin{equation}\label{saddle m lemma eq 0}
\sum_{\substack{m\in\mathcal{E}\\ M_1<m<M_2}} 
\left|\Hi(\omega,N;\mu,s_m)-\Hi(\mu,s_m)\right|\le 
\frac{16}{\pi}q +  \frac{4}{\pi}q\log(2M-1)\mathds{1}_{M>0}.
\end{equation}
If $\omega < -N$, then the same bound holds but with $\overline{\Hi(\mu,s_m)}$
instead of $\Hi(\mu,s_m)$.
\end{lemma}
\begin{proof}
Assume that $\omega >0$. In view of the identity
\begin{equation}
\Hi(\omega,N;\mu,s_m)=\Hi(\mu,s_m)-\Hi(0,\omega;\mu,s_m)-\Hi(\omega+N,\infty;\mu,s_m),
\end{equation}
the l.h.s. in \eqref{saddle m lemma eq 0} is  
\begin{equation}
\le\sum_{\substack{m\in\mathcal{E}\\ M_1<m<M_2}} |\Hi(0,\omega;\mu,s_m)|
+\sum_{\substack{m\in\mathcal{E}\\ M_1<m<M_2}} |\Hi(\omega+N,\infty;\mu,s_m)|.
\end{equation}
We treat the sum involving $\Hi(0,\omega;\mu,s_m)$ first. 
Appealing to Lemma 4.2. in \cite{titchmarsh},
\begin{equation}\label{saddle m lemma eq 1}
|\Hi(0,\omega;\mu,s_m)| \le 
\frac{2}{\pi}\max_{0\le t\le \omega} \frac{1}{3|\mu t^2-s_m|}.
\end{equation}
For all $m \in (M_1,M_2)\cap \mathcal{E}$
either $m = M_1+2j+1$ or $m=M_1+2j+2$ throughout, 
where $j$ a nonnegative integer.
The  correct parity is determined by $\mathcal{E}$. 
Also, since $\omega > 0$, $M_1 = [2 q \alpha]\ge 2q\alpha - 1/2$. 
Therefore, in any case, we have $m\ge M_1+2j+1$ throughout and 
$s_m = \mu \omega^2+(m/2-q\alpha)/3q 
\ge \mu \omega^2 + (j+1/4)/3q$.
Combining this with 
the trivial 
bound $\max_{0\le t\le \omega} |\mu t^2| = \mu \omega^2$ 
and inserting into \eqref{saddle m lemma eq 1} now  
yields 
\begin{equation}\label{saddle m lemma eq 2}
\sum_{\substack{m\in\mathcal{E}\\ M_1<m<M_2}}
|\Hi(0,\omega;\mu,s_m)| \le \frac{2}{\pi} \sum_{0\le j \le \frac{M_2-M_1-1}{2}}
\frac{q}{j+1/4}.
\end{equation}

As for the sum involving $\Hi(\omega+N,\infty;\mu,s_m)$, 
we write $t=\omega+u$ with $u\ge N$. Then, like before, 
we apply Lemma 4.2 in \cite{titchmarsh} to obtain
\begin{equation}\label{saddle m lemma eq 3}
|\Hi(\omega+N,\infty;\mu,s_m)| \le \frac{2}{\pi}\max_{t \ge  \omega +N}
\frac{1}{3|\mu t^2-s_m|} 
= \frac{2}{\pi}\max_{u\ge N} \frac{2 q}{|2qf'(u)-m|}.
\end{equation}
We proceed analogously to the previous sum.
Specifically, for all $m\in (M_1,M_2)\cap \mathcal{E}$ either
$m = M_2-2j-1$ or $m=M_2-2j-2$ where $j$ is a nonnegative integer. 
Also, $M_2 = [2qf'(N)]\le 2qf'(N)+1/2$.
Hence, $m \le 2qf'(N) -2j-1/2$. 
Last, using that $\omega > 0$ (so the minimum of $f'(u)$ occurs when $u=-\omega
<0$), we obtain that $\min_{u\ge N} 2qf'(u) \ge 2qf'(N)$. 
Therefore, put together, we conclude that 
\begin{equation}\label{saddle m lemma eq 4}
\sum_{\substack{m\in\mathcal{E}\\ M_1<m<M_2}}
|\Hi(\omega+N,\infty;\mu,s_m)| \le \frac{2}{\pi} \sum_{0\le j \le
\frac{M_2-M_1-1}{2}}
\frac{q}{j+1/4}.
\end{equation}
The sums in \eqref{saddle m lemma eq 2} and 
\eqref{saddle m lemma eq 4} 
are bounded routinely. If $M_1=M_2$, then these sums are empty. 
And if $M_1<M_2$, then one isolates the term for $j=0$ 
and compares the remaining sum to an integral. 
Putting these bounds together yields
the lemma when $\omega > 0$.

The treatment of the case $\omega < -N$ 
is analogous
except one starts with the identity
\begin{equation}
\Hi(\omega,N;\mu,s_m)=\overline{\Hi(\mu,s_m)}
+\Hi(\omega,-\infty;\mu,s_m)
+\Hi(0,\omega+N;\mu,s_m),
\end{equation}
then continues as in the previous case, this time appealing to the bounds
\begin{equation}
\begin{split}
&|\Hi(\omega,-\infty;\mu,s_m)| \le 
\frac{2}{\pi}\max_{t\ge |\omega|} \frac{1}{3|\mu t^2-s_m|},\\
&|\Hi(0,\omega+N;\mu,s_m)|  
\le \frac{2}{\pi}\max_{0\le t \le |\omega+N|} \frac{1}{3|\mu t^2-s_m|}
= \frac{2}{\pi}\frac{2q}{|2qf'(N)-m|},
\end{split}
\end{equation}
and the formulas $M_1 =[2qf'(N)]$ and $M_2= [2q \alpha]$, 
 valid for $\omega < -N$.
To handle the integral $|\Hi(0,\omega+N;\mu,s_m)|$, one additionally uses that
 $|\omega+N| = |\omega|-N$ combined with 
 the change of variable $t\leftarrow |\omega|-u$, $N\le u\le |\omega|$,
and observation that 
$3\mu(|\omega|-u)^2-s_m=f'(u)-m/(2q)$ is 
decreasing in $u$ over $N \le u\le |\omega|$. 
\end{proof}

\begin{lemma}\label{saddle m lemma 2}
If $-N/2 < \omega \le 0$, then
\begin{equation}\label{saddle m lemma 2 eq 1}
\sum_{\substack{m\in\mathcal{E}\\ M^*<m<M_2}} 
\left|\Hi(\omega,N;\mu,s_m)-\Hi(\mu,s_m)\right|\le 
\frac{16}{\pi}q +  \frac{4}{\pi}q\log(2M-1)\mathds{1}_{M>0}.
\end{equation}
If $-N\le \omega \le -N/2$, then the same bound holds but with
$\Hi(\mu,s_m)$ replaced by its conjugate $\overline{\Hi(\mu,s_m)}$. 
\end{lemma}

\begin{proof}
The proof of the first bound, i.e.\
 when $-N/2 < \omega \le 0$, follows 
analogously as in the proof of Lemma~\ref{saddle m
lemma} for the case $\omega >0$. The proof of the second bound, 
 i.e.\ when $-N\le \omega \le -N/2$, also follows as in Lemma~\ref{saddle m
lemma} but for the case $\omega <-N$.
\end{proof}

\section{Terms with two saddle points}\label{two saddle pts}

\begin{lemma}\label{saddle m lemma 3}
If $-N\le \omega \le 0$, then
\begin{equation}\label{saddle m lemma 2 eq 0}
\sum_{\substack{m\in\mathcal{E}\\ M_1<m<M^*}} 
\left|\Hi(\omega,N;\mu,s_m)-\Ai(\mu,s_m)\right|\le 
\frac{16}{\pi}q +  \frac{4}{\pi}q\log(2M-1)\mathds{1}_{M>0}.
\end{equation}
\end{lemma}
\begin{proof}
We start with the identity
\begin{equation}
\Hi(\omega,N;\mu,s_m) = \Ai(\mu,s_m)
+\Hi(\omega,-\infty;\mu,s_m)-\Hi(\omega+N,\infty;\mu,s_m).
\end{equation}
Let us first recall that
$s_m =\mu\omega^2+ (m/2-q\alpha)/(3q)$.
Also, $M^* \le 2q\alpha +1/2$ if $\omega \ge -N/2$, 
$M^* \le 2qf'(N)+1/2$ if $\omega < -N/2$, and 
$f'(N) \le f'(0)=\alpha$ if $-N\le \omega < -N/2$.
So we deduce, in all cases, that $m/2-q\alpha
< 0$ for $m<M^*$ and in particular 
$s_m < \mu \omega^2 = \min_{t\ge |\omega|} \mu t^2$.

Now, applying Lemma 4.2 in \cite{titchmarsh} to 
each term $\Hi(\omega,-\infty;\mu,s_m)$ gives
\begin{equation}
\begin{split}
\sum_{\substack{m\in\mathcal{E}\\ M_1<m<M^*}} 
|\Hi(\omega,-\infty;\mu,s_m)| &\le 
 \frac{2}{\pi}\sum_{\substack{m\in\mathcal{E}\\ M_1<m<M^*}}
\frac{2q}{|m-2q\alpha|}.
\end{split}
\end{equation}
Let $m=M^*-2j-1$ with $j\in \mathbb{Z}_{\ge 0}$. 
By the previous observations about $M^*$, we obtain
$|2q\alpha-m| \ge 2j+1/2$, hence the last sum is 
\begin{equation}
\le \frac{2}{\pi}\sum_{0\le j\le \frac{M^*-M_1-1}{2}} \frac{q}{j+1/4}.
\end{equation}
We estimate this sum by an integral, as was done for \eqref{saddle m lemma eq 4}, 
which verifies the bound in the lemma. 
\end{proof}

\begin{remark}   
    We have 
$$\Ai(\mu,s) =
    \frac{2\pi}{(6\pi\mu)^{1/3}}\textrm{Ai}\left(-\frac{(2\pi)^{2/3}s}{(3\mu)^{1/3}}\right)$$
where
    $\textrm{Ai}(x):=\frac{1}{2\pi}\int_{-\infty}^{\infty}e^{it^3/3+ixt}dt$
is the usual Airy function satisfying $|\textrm{Ai}(x)| \le
1/|x|^{1/4}$, and so one obtains $|\Ai(\mu,s)| \le
        \sqrt{2\pi}/(3\mu|s|)^{1/4}$. 
\end{remark}

\section{An alternative bound for the tail}\label{alternative bound}

    We may consider the tails $\Upsilon_1$ and $\Upsilon_2$ in
Section~\ref{no saddle pts} together, and apply the
method of Lemma~\ref{large m lemma} to both of them. 
This has the effect of adding more terms to the function
$\tilde{\Phi}$ in Section~\ref{no saddle pts} and gives an error term that still goes to zero as
$\beta$ and $\mu$ go to zero but with an extra factor of $q^3$. 

Explicitly,  rather than apply  Lemma 4.2 in \cite{titchmarsh} 
to each term in $\Upsilon_2$ immediately, 
we first apply integration by parts followed by
an application of Lemma 4.2 in \cite{titchmarsh} 
then proceed similarly to the proof of Lemma~\ref{large m lemma}.
This yields the following bound: Define 
\begin{equation}\label{simple Phi 1}
    \Phi^o_{\omega+x} = \frac{qe(f(x))}{\pi
i c_0}\sum_{\substack{m\in\mathcal{E}\\
m\not\in [M_1,M_2]}}  \frac{e(g(m)-xm/2q)}{2qf'(x)-m}.
\end{equation}
which similar to \eqref{Phi omega} except it involves the additional terms $m\in
(M_1,M_1+q]$ and $m\in (M_2,M_2+q]$. 

\begin{lemma}\label{tail m lemma}
$\displaystyle \Big|\Upsilon_1+\Upsilon_2
-\Phi^o_{\omega+N}+ \Phi^o_{\omega}\Big| \le
\frac{576(2|\beta|+3\mu N)q^3}{\pi^2}$.
\end{lemma}

One may also use integration by parts to execute the proofs
of the lemmas in Section~\ref{one saddle pt} and 
Section~\ref{two saddle pts}. This would add yet more terms to the
function $\Phi^o_{\omega+x}$, enlarging the range of summation to all $m\not\in
\Omega = \{ M_1,M_2,M^*\}$.

\begin{lemma}
If $\omega > -N/2$, then
\begin{equation}
\begin{split}
\sum_{\substack{m\in\mathcal{E}\\ M^*<m<M_2}} 
& \left|\Hi(\omega,N;\mu,s_m)-\Hi(\mu,s_m)-
\phi_{\omega+N}(\mu,s_m)+\phi_{\omega}(\mu,s_m)-\phi_0(\mu,s_m)\right| \le \\ 
&\frac{576(2|\beta|+3\mu N)q^3}{\pi^2}.
\end{split}
\end{equation}
If $\omega \le -N/2$, then the same bound holds but with $\Hi(\mu,s_m)$
replaced by its conjugate and $-\phi_0(\mu,s_m)$ replaced by $\phi_0(\mu,s_m)$. 
\end{lemma}

\begin{lemma}
If $-N\le \omega \le 0$, then
\begin{equation}
\begin{split}
\sum_{\substack{m\in\mathcal{E}\\ M_1<m<M^*}} 
& \left|\Hi(\omega,N;\mu,s_m)-\Ai(\mu,s_m)-
\phi_{\omega+N}(\mu,s_m)+\phi_{\omega}(\mu,s_m)\right| \le \\ 
&\frac{576(2|\beta|+3\mu N)q^3}{\pi^2}.
\end{split}
\end{equation}
\end{lemma}

In view of the previous two lemmas, we are motivated to define 
\begin{equation}\label{simple Phi 2}
    \Phi(x) := \frac{qe(f(x))}{\pi
i c_0}\sum_{\substack{m\in\mathcal{E}\\ m\not\in \Omega}}  \frac{e(g(m)-xm/2q)}{2qf'(x)-m},
\end{equation}
which accounts for the terms $\phi_{\omega}$ and $\phi_{\omega+N}$, and
\begin{equation}
Y(x) := \frac{\sgn(x+N/2)}{6\pi i} \sum_{\substack{m\in\mathcal{E}\\M^*<m<M_2}}\frac{e(g(m)+x m/2q)}{s_m},
\end{equation}
which accounts for the term $\phi_0$. 
The numerator in these definitions is inserted because $\phi_x$ will be 
multiplied by $e(g(m)+\omega m/2q)$ according to the formula in Lemma~\ref{C formula}; 
see \eqref{B formula eq}.

\section{Formulas for the transformed sum}\label{summary formula 1}

In summary, we have proved the following.
Define
\begin{equation}
T_m:=e\left(\frac{\omega m}{2q}+g(m)\right)\Hi(\mu,s_m),
\end{equation}
and define $\bar{T}_m$ the same way as $T_m$ except that $\Hi(\mu,s_m)$ is replaced by its
conjugate while $e(\omega m/2q+g(m))$ is kept the same.  
Moreover, define the boundary term
\begin{equation}
    \mathcal{B} := \sum_{\substack{\ell\in \mathcal{E} \\ \textrm{distinct $\ell
    \in \Omega$}}} B_{\ell}.
\end{equation}
To clarify the behavior of the main sum $\mathcal{M}$ below, 
we refer to Lemmas \ref{basic lemma 1} \& \ref{basic lemma 2}.
Also, estimates for $\mathcal{B}$ are provided in Lemma~\ref{basic lemma 3}.

\begin{proposition}\label{main theorem}
\begin{equation}
    C(N;a,b,q;f)= \frac{c_1}{\sqrt{q}} \left[\mathcal{M}
+\mathcal{B}+\mathcal{R}_1\right]+ \frac{1+B}{2},
\end{equation}
where the main sum $\mathcal{M}$ is equal to 
\begin{equation}
\begin{split}
\mathcal{M}=
\displaystyle \sum_{\substack{m\in\mathcal{E}\\ M_1<m<M^*}} (T_m+\bar{T}_m)  
+\left\{\begin{array}{ll}
\displaystyle \sum_{\substack{m\in\mathcal{E}\\ M^*<m<M_2}}  T_m, &\textrm{if $\omega
\ge -N/2$},\\ 
&\\
\displaystyle \sum_{\substack{m\in\mathcal{E}\\ M^*<m<M_2}}  \bar{T}_m, &\textrm{if $\omega
\le -N/2$}, 
\end{array}\right.
\end{split}
\end{equation}
where the remainder term $\mathcal{R}_1$ satisfies the bound
\begin{equation}
\begin{split}
|\mathcal{R}_1| &\le 
\frac{128(2|\beta|+3\mu N)}{\pi^2} \frac{q^3(2q+9)}{(2q+1)^3}
+ \frac{4q}{\pi} \log\left(\frac{M_{\max}}{q+1/2}+1\right)\\
&+\frac{8}{\pi}q \log(2q-1)+ \frac{8}{\pi}q\log(2M-1)\mathds{1}_{M>0}
+\frac{92}{\pi}q.
\end{split}
\end{equation}
\end{proposition}

Note that the remainder $\mathcal{R}_1$ satisfies 
$\mathcal{R}_1 \ll
q(|\beta|+\mu N+  \log M_{\max}+ \log2q)$;
in particular, $\mathcal{R}_1$ does not tend to zero as $\beta$ and $\mu$ tend to zero. 
However, by incorporating more lower order terms 
using the lemmas in Section~\ref{alternative bound} 
we can obtain a remainder term that tends to zero with $\beta$ and $\mu$ but
that depends more heavily on $q$; namely, 
we obtain a remainder of size $\ll (|\beta|+\mu N)q^3$.

\begin{proposition}\label{main theorem 2}
\begin{equation}
    C(N;a,b,q;f)= \frac{c_1}{\sqrt{q}} \left[\mathcal{M}
+\mathcal{B}+\Phi(N)-\Phi(0)+Y(\omega)+\mathcal{R}_2\right]+ \frac{1+B}{2},
\end{equation}
and
\begin{equation}
|\mathcal{R}_2|\le 
\frac{1728(2|\beta|+3\mu N)q^3}{\pi^2}.
\end{equation}
\end{proposition}

\begin{lemma} \label{basic lemma 1}
If $s> 0$, then
\begin{equation}
\left|\Hi(\mu,s) - 
\frac{1}{(36\mu s)^{1/4}}e\left(\frac{1}{8}-\frac{2s\sqrt{s}}{\sqrt{\mu}}\right)
\right|\le \frac{1}{\pi s}. 
\end{equation}
\end{lemma}
\begin{proof}
    By a change of variable $t\leftarrow \mu^{1/3}t$, we obtain
\begin{equation}
    \Hi(\mu,s)  = \frac{1}{\mu^{1/3}}\int_0^{\infty} e(t^3-3st/\mu^{1/3})dt. 
\end{equation}
A close examination of the proof of \cite[Lemmas 2.5 \& 2.6]{bombieri-iwaniec-1} (applied
with $y=s/\mu^{1/3}$) gives the result.
\end{proof}

\begin{lemma}  \label{basic lemma 2}
    If $m>M_1$ then $s_m>0$.  Specifically, 
$$
s_m \ge 
\left\{\begin{array}{ll}
        \vspace{1mm}
\displaystyle \mu\omega^2+ \frac{m-M_1-1/2}{6q}  &\textrm{if $\, \omega > 0$},\\ 
        \vspace{1mm}
        \displaystyle \frac{m-M_1-1/2}{6q} &\textrm{if $\, -N\le \omega \le 0$},\\ 
\displaystyle \mu(\omega+N)^2+\frac{m-M_1-1/2}{6q} &\textrm{if $\, \omega < -N$}. 
    \end{array}\right.
$$
\end{lemma}
\begin{proof}
This follows from the definitions of $M_1,M_2$ and
$M^*$.
\end{proof}

\begin{lemma}\label{basic lemma 4}
Suppose that $q=1$, $a=0$ and $3|m-\alpha|\mu/\beta^2 \le 1-\epsilon_1 <1$. 
If $\beta >0$ or $\beta < -1/N$ then 
\begin{equation*}
\begin{split}
\frac{2\beta^3}{27\mu^2}-\frac{\beta\alpha}{3\mu}+ \frac{\beta m}{3\mu}-&\sgn(\beta)\frac{2s_{2m}\sqrt{s_{2m}}}{\sqrt{\mu}} = -\frac{\alpha^2}{4\beta}-\frac{\alpha^3 \mu}{8\beta^3}
+\left(\frac{\alpha}{2\beta}+\frac{3\alpha^2\mu}{8\beta^3}\right)m\\
&+\left(-\frac{1}{4\beta}-\frac{3\alpha \mu}{8\beta^3}\right)m^2+\frac{\mu}{8\beta^3}m^3
+ O_{\epsilon_1}\left(\frac{\mu^2 |m-\alpha|^4}{\beta^5}\right).
\end{split}
\end{equation*}
Moreover,
\begin{equation}
\frac{1}{(36\mu s_{2m})^{1/4}} = \frac{1}{\sqrt{2|\beta|}}+
O_{\epsilon_1}\left(\frac{1}{\sqrt{|\beta|}}\frac{\mu
|m-\alpha|}{\beta^2}\right).
\end{equation}
\end{lemma}

\begin{lemma}  \label{basic lemma 3}
    We have
    $$
|\mathcal{B}|\le 
\left\{\begin{array}{ll}
        \vspace{1mm}
        \min\left(2N,\displaystyle\frac{16}{\sqrt{12\pi \mu \omega}}\right)      &\textrm{if $\, \omega > 0$},\\             
        \vspace{1mm}
        \min\left(3N,\displaystyle\frac{48}{(12\pi \mu)^{1/3}}\right)     &\textrm{if $\, -N\le
    \omega \le 0$},\\             
    \min\left(2N,\displaystyle\frac{16}{\sqrt{-12\pi \mu (\omega+N)}}\right)         &\textrm{if $\, \omega < -N$}. 
\end{array}\right.
$$
\end{lemma}
\begin{proof}
    The bounds when $\omega >0$ or $\omega <-N$ 
    follow from Lemma 4.4 in \cite{titchmarsh} and on considering 
    that two terms contribute to $\mathcal{B}$ in these cases.
    When $-N\le \omega \le 0$, there are at most three terms contributing to
    $\mathcal{B}$. 
    Write $\Hi(\omega,N;\mu,s) =
    \int_{\omega}^0 e(\mu t^3-3s)dt + \int_0^{\omega+N} e(\mu t^3-3s)dt$ 
    then treat each integral separately; e.g.\ 
    $\int_{\omega}^0 e(\mu t^3-3s)dt = \int_{\omega}^{\delta} e(\mu t^3-3s)dt + 
    \int_{\delta}^0 e(\mu t^3-3s)dt$, bound the first integral using
    Lemma 4.4 in \cite{titchmarsh} and bound the second integral trivially,
    then optimize the choice of $\delta = 4/(12\pi \mu)^{1/3}$.
\end{proof}

\begin{proposition}\label{intro theorem 1}
    Let $w:=[6\mu qN^2]$. 
    If $w\ne 0$ then 
\begin{equation*}
\begin{split}
  &  \sideset{}{'}\sum_{n=0}^N e\left(\frac{an+bn^2}{2q}+\mu n^3\right)=\\
 &\frac{e(1/8)g(b+\delta q,q)}{(6\mu q)^{1/4}} 
    \sum_{\substack{0 < m < w\\ m\equiv \delta_1 \pmod*2}}
    \frac{1}{m^{1/4}}e\left(\frac{b^*(a+m)^2}{8q}
    -\frac{2m^{3/2}}{6q\sqrt{6\mu q}}\right)\\
    &+ \mathds{1}_{\delta_1=0}\,\frac{g(b+\delta q,q)}{\sqrt{q}} 
    e\left(\frac{b^* a^2}{8q}\right)\int_0^N e(\mu t^3)dt\\
    &+\mathds{1}_{\delta_1\equiv w\pmod*2}\,
    \frac{g(b+\delta q,q)}{\sqrt{q}} 
    e\left(\frac{b^* (a+w)^2}{8q}\right)
    \int_0^N e\left(\mu
     t^3-\frac{wt}{2}\right)dt\\
    &+O\left(\mu N q^{1/2} +q^{1/2}\log (w+2q)\right)
 \end{split}
 \end{equation*}
 where the prime on the sum means that
 the boundary terms at $n=0$ and $N$ are weighted by $1/2$.
 If $w=0$, i.e.\ if $\mu < 1/(12qN^2)$, then the two integrals on the
 r.h.s.\ are equal and one of them is dropped. 
 \end{proposition}
 \begin{proof}
Apply Proposition~\ref{main theorem} with $\alpha = \beta =
0$, followed by lemmas \ref{basic lemma 1} \& \ref{basic lemma 2}.
 \end{proof}

\section{Proofs}\label{proofs}

 \begin{proof}[Proof of Theorem~\ref{intro theorem 2}]
     This is a special case of Proposition~\ref{main theorem} when the intervals
     $(M_1,M^*)$ and $(M^*, M_2)$ are empty, so $\mathcal{M}=0$ and the only terms that survive are
     the boundary terms $\mathcal{B}$. 
 \end{proof}

 \begin{proof}[Proof of Theorem~\ref{intro theorem 3}]
Consider the case  $q=1$ and $a=0$. Then necessarily $b=0$
and $C(N;a,b,q;f) = H_N(\alpha,\beta,\mu)$.
Moreover, $b^*=0$, $b+\delta q= 0$ and
$m\in \mathcal{E}$ is equivalent to $m\in 2\mathbb{Z}$. 
Therefore, in the situation  $q=1$ and $a=0$,
we have $g(m)\equiv 0$, $G(0,b+\delta q;2q)=1$ 
and we need only consider even $m$ in Proposition~\ref{main theorem}. In addition, 
since $g(m)\equiv 0$, we have $c_1 = c_0 =
e(2\beta^3/27\mu^2-\beta\alpha/3\mu)$.  Suppose further that $|\beta| >1/N$ 
 and that $0<6 N^2\mu < 1$. Then $|\omega| >N$ 
and so $M^*=M_1$. In particular, Proposition~\ref{main theorem} involves only
$T_{2m}$ if $\beta >1/N$, and only $\bar{T}_{2m}$ if $\beta < -1/N$. Therefore,
after simplifying using lemmas \ref{basic lemma 1} and \ref{basic lemma 2}, we see that the terms that need considered in Proposition~\ref{intro theorem 1}  are of the form
$$\frac{e(2\beta^3/27\mu^2-\beta\alpha/3\mu)}{(36\mu s_{2m})^{1/4}}e\left(\frac{\beta m}{3\mu}+\frac{\sgn(\beta)}{8}-\sgn(\beta)\frac{2s_{2m}\sqrt{s_{2m}}}{\sqrt{\mu}}\right) +O\left(\frac{1}{s_{2m}}\right).$$
This motivates considering the Taylor series appearing in Lemma~\ref{basic
lemma 4}. Note that the conditions required by this lemma 
are satisfied due to our assumptions on
$\mu$ and $\beta$. Indeed, if we substitute these expansions into Lemma~\ref{basic lemma 2} 
then back into Proposition~\ref{main theorem}, 
    and use Lemma~\ref{basic lemma 3} to estimate the boundary terms
    $\mathcal{B}$,
    then we obtain the result. 
 \end{proof}

 \section{Suggested improvements}\label{improvements}

 One might be able to remove the $\log(|N'|+2)$ term appearing in the
 $O$-notation in Theorem~\ref{intro theorem 3} 
by using Proposition~\ref{main theorem 2} instead of
Proposition~\ref{main theorem} to derive the theorem. The former proposition 
incorporates secondary terms $\Phi(x)$ and $Y(x)$ 
which may be estimated more precisely and it has
a remainder $\mathcal{R}_2$ that tends to zero with $\beta$ and $\mu$. 
Similarly, one might be able to remove $\log(2q)$ factor 
from the remainder in
Theorem~\ref{intro theorem 2} by using Proposition~\ref{main
theorem 2} instead of
Proposition~\ref{main theorem}. Both improvements will
require careful and substantial analysis of the functions $\Phi(x)$ and $Y(x)$.
e.g.\ one probably should divide the sum in $\Phi(x)$ along arithmetic
progressions modulo $2q$ so as to express $\Phi(x)$ as a linear combination of 
    Hurwitz--Lerch zeta functions then apply known asymptotics for the latter.

Additionally, it might be desirable to derive a version of the bound \eqref{H
bound} where instead of $H_N^{\max}(\alpha,\beta,\mu)$ we use the function
\begin{equation}
    \max_{N_2 \in [0,N]}\left|\sum_{n=0}^{N_2} e(\alpha n+\beta n^2+\mu
    n^3)\right|,
\end{equation}
which offers some advantages; 
e.g.\ if we start with $\alpha=0$ then new $\alpha$ (in the transformed sum)
will still be zero.
Finally, although we have not done so for 
the results stated in the introduction, all the implicit constants
appearing there can be made explicit if desired by using 
 the explicit error bounds in Section~\ref{summary formula 1}.

\bibliographystyle{amsplain}
\bibliography{cubicExp}
\end{document}